\documentclass{groups13}
\usepackage{amssymb}
\usepackage[curve,matrix,arrow]{xy}

\newenvironment{proofscottreg}{\begin{trivlist}\item[]
\textbf{Proof of Theorem \ref{Scottregzero}}\ }{\hspace*{\fill} $\Box$\end{trivlist}}

\newenvironment{proofhhreg}{\begin{trivlist}\item[]
\textbf{Proof of Theorem \ref{HHregzero}}\ }{\hspace*{\fill} $\Box$\end{trivlist}}

\newenvironment{prooffreg}{\begin{trivlist}\item[]
\textbf{Proof of Theorem \ref{Fregzero}}\ }{\hspace*{\fill} $\Box$\end{trivlist}}

\def\CF{{\mathcal{F}}}

           \def\tenk{\otimes_k}     
\def\Br{\mathrm{Br}}             

           \def\tenkP{\otimes_{kP}}
           \def\tenkQ{\otimes_{kQ}}
\def\ev{\mathrm{ev}}  
\def\Ext{\mathrm{Ext}}           

\def\Hom{\mathrm{Hom}}           
\def\ker{\mathrm{ker}}           
             
\def\Im{\mathrm{Im}}             
\def\Ind{\mathrm{Ind}}

\def\max{\mathrm{max}}

\def\op{\mathrm{op}}
\def\pr{\mathrm{pr}}

\def\reg{\mathrm{reg}}
\def\Res{\mathrm{Res}}           
\def\res{\mathrm{res}}        
\def\Sol{\mathrm{Sol}}          
           
\def\Tr{\mathrm{Tr}}             
\def\tr{\mathrm{tr}}

\begin{document}
\maketitle{On the Castelnuovo-Mumford regularity of the cohomology of 
fusion systems and of the Hochschild cohomology 
of block algebras} {Radha Kessar and Markus Linckelmann}{Department of Mathematical  Sciences,  City University London,
Northampton Square, London EC1V 0HB, 
U.K. \newline 
Email:  radha.kessar.1@city.ac.uk, markus.linckelmann.1@city.ac.uk}

\begin{abstract} 
Symonds' proof of Benson's regularity conjecture implies that
the regularity of the cohomology of a fusion system and that of
the Hochschild cohomology of a $p$-block of a finite group
is at most zero. Using results of Benson, Greenlees, and Symonds, 
we show that in both cases the regularity is equal to zero.
\end{abstract}

\spc

Let $p$ be a prime and $k$ an algebraically closed field of 
characteristic $p$. Given a finite group $G$, a {\it block algebra of} 
$kG$ is an indecomposable direct factor $B$ of $kG$ as a $k$-algebra.
A {\it defect group of of a block algebra $B$ of $kG$} is a minimal 
subgroup $P$ of $G$ such that $B$ is isomorphic to a direct 
summand of $B\tenkP B$ as a $B$-$B$-bimodule. The defect groups of 
$B$ form a $G$-conjugacy class of $p$-subgroups of $G$.
The Hochschild cohomology of $B$ is the algebra $HH^*(B)=$
$\Ext^*_{B\tenk B^\op}(B)$, where $B^\op$ is the opposite algebra of
$B$, and where $B$ is regarded as a $B\tenk B^\op$-module via left and
right multiplication. By a result of Gerstenhaber, the algebra
$HH^*(B)$ is graded-commutative; that is, for homogeneous elements
$\zeta\in$ $HH^m(B)$ and $\eta\in$ $HH^m(B)$ we have $\eta\zeta=$
$(-1)^{nm}\zeta\eta$, where $m$, $n$ are nonnegative integers. 
In particular, if $p=$ $2$, then $HH^*(B)$ is commutative, and if
$p$ is odd, then the even part $HH^\ev(B)=$ $\oplus_{n\geq 0} HH^{2n}(B)$
is commutative and all homogeneous elements in odd degrees square to zero.  
The extension of the Castelnuovo-Mumford regularity to
graded-commutative rings with generators in arbitrary positive
degrees is due to Benson \cite[\S 4]{Benreg1}. We follow
the notational conventions in Symonds \cite{Symreg}.
In particular, if $p$ is odd and 
$T=$ $\oplus_{n\geq 0} T^n$ is a finitely generated 
graded-commutative $k$-algebra and $M$ a finitely generated graded 
$T$-module, we denote by $\reg(T,M)$ the Castelnuovo-Mumford
regularity of $M$ as a graded $T^{\ev}$-module, where
$T^{\ev}=\oplus_{n\geq 0} T^{2n}$ is the even part of $T$. 
We set $\reg(T)=$ $\reg(T,T)$; that
is, $\reg(T)$ is the Castelnuovo-Mumford regularity of $T$ as a
graded $T^{\ev}$-module. See also \cite{Bencomm} and \cite{Eisenbud}
for more background material and references.  We note  that  Benson's definition of regularity   uses 
the ring $T$ instead of $T^{\ev}$, but the two  definitions  are  
equivalent. This can be seen  by  noting that  \cite[Proposition  1.1]{Symreg} also holds for  finitely generated graded commutative $k$-algebras.

\begin{thm} \label{HHregzero}
Let $G$ be a finite group and $B$ a block algebra of $kG$.
We have $\reg(HH^*(B))=$ $0$.
\end{thm}

This will be shown as a consequence of a statement on Scott modules.
Given a finite group $G$ and a $p$-subgroup $P$ of $G$, there
is up to isomorphism a unique indecomposable $kG$-module $Sc(G;P)$ with
vertex $P$ and trivial source having a quotient (or equivalently, a
submodule) isomorphic to the trivial $kG$-module $k$. The module
$Sc(G;P)$ is called the {\it Scott module of $kG$ with vertex $P$}.
It is constructed as follows: Frobenius reciprocity implies that
$\Hom_{kG}(\Ind^G_P(k),k)\cong$ $\Hom_{kP}(k,k)\cong$ $k$, and hence
$\Ind_P^G(k)$ has up to isomorphism a unique direct summand $Sc(G;P)$ 
having $k$ as a quotient. Since $\Ind^G_P(k)$ is selfdual, the 
uniqueness of $Sc(G;P)$ implies that $Sc(G;P)$ is also selfdual, and 
hence $Sc(G;P)$ can also be characterised as the unique summand, up to 
isomorphism, of $\Ind_P^G(k)$ having a nonzero trivial submodule. 
Moreover, it is not difficult to see that $Sc(G;P)$ has $P$ has a 
vertex. See \cite{BroueScott} for more details on Scott modules, as well 
as \cite{KeKuMi} for connections between Scott modules and fusion 
systems. For a finitely generated graded module $X$ over $H^*(G;k)$ we 
denote by $H_m^{*,*}(X)$ the local cohomology with respect to the 
maximal ideal of $H^*(G;k)$ generated by all elements in positive 
degree. The first grading is here the local cohomological grading, and 
the second is induced by the grading of $X$.

\begin{thm} \label{Scottregzero}
Let $G$ be a finite group and $P$ a $p$-subgroup of $G$. We have 
$$\reg(H^*(G;k); H^*(G;Sc(G;P)))=0\ .$$
\end{thm}

\begin{re}
Using Benson's reinterpretation in \cite[\S 4]{BenInjCohom}, of the `last 
survivor' from \cite[\S 7]{BeCa}, applied to the Scott module instead 
of the trivial module, one can show more precisely that 
$$H_m^{r,-r}(H^*(G;Sc(G,P)))\neq \{0\}\ ,$$ 
where $r$ is the rank of $P$. 
It is not clear whether this property, or even the property of having 
cohomology with regularity zero, characterises Scott modules amongst 
trivial source modules.
\end{re}

For $\CF$ a saturated fusion system on a finite $p$-group $P$, we 
denote by $H^*(P;k)^\CF$ the graded subalgebra of $H^*(P;k)$ 
consisting of all elements $\zeta$ satisfying $\Res^P_Q(\zeta)=$
$\Res_\varphi(\zeta)$ for any subgroup $Q$ of $P$ and any
morphism $\varphi : Q\to$ $P$ in $\CF$. If $\CF$ is the fusion
system of a finite group $G$ on one of its Sylow-$p$-subgroups
$P$, then $H^*(P;k)^\CF$ is isomorphic to $H^*(G;k)$ through
the restriction map $\Res^G_P$, by the characterisation of $H^*(G;k)$
in terms of stable elements due to Cartan and Eilenberg. In that
case we have $\reg(H^*(P;k)^\CF)=$ $0$ by \cite[Corollary 0.2]{Symreg}.
If $\CF$ is the fusion system of a block algebra $B$ of $kG$ on a
defect group $P$, then $H^*(P;k)^\CF$ is the block cohomology $H^*(B)$
as defined in \cite[Definition 5.1]{LinHH}. It is not known whether
all block fusion systems arise as fusion systems of finite groups.
There are examples of fusion systems which arise neither from
finite groups nor from blocks; see \cite{KeSol}, \cite{KeSta}.

\begin{thm} \label{Fregzero}
Let $\CF$ be a saturated fusion system on a finite $p$-group $P$.
We have $$\reg(H^*(P;k)^\CF)=0\ .$$
\end{thm}

The key ingredients for proving the above results are Greenlees' local 
cohomology spectral sequence \cite[Theorem 2.1]{Greenlees}, results and
techniques in work of Benson \cite{BenInjCohom}, \cite{Benreg1},
\cite{Benreg2}, and Symonds' proof in \cite{Symreg} of Benson's 
regularity conjecture. We use the properties of the regularity from
\cite[\S 1]{Symreg} and \cite[\S 2]{Symregpol}. 

\begin{lemma} \label{regtrivialsource}
Let $G$ be a finite group and $V$ an indecomposable trivial source 
$kG$-module. Then $\reg(H^*(G;k); H^*(G; V)) \leq$ $0$.
\end{lemma}

\begin{proof}
Since $V$ is a direct summand of $\Ind^G_P(k)$, we have
$$\reg(H^*(G;k); H^*(G; V)) \leq \reg(H^*(G;k); H^*(G; \Ind^G_P(k))\ .$$
By \cite[Lemma 4]{KeLiHHgen}, the right side is
equal to $\reg(H^*(P;k))$, hence zero by \cite[Corollary 0.2]{Symreg}.
\end{proof}

\begin{lemma} \label{regnonzerohomology}
Let $G$ be a finite group and $V$ a finitely generated $kG$-module.
If $H_0(G;V)\neq$ $\{0\}$, then $\reg(H^*(G;k); H^*(G; V)) \geq$ $0$.
\end{lemma}

\begin{proof}
It follows from the assumption $H_0(G;V)\neq$ $\{0\}$ and
Greenlees' spectral sequence \cite[Theorem 2.1]{Greenlees}
that there is an integer $s$ such that $H_m^{s,-s}(H^*(G;V))\neq$ $\{0\}$,
which implies the result.
\end{proof}

\begin{proofscottreg} Set $V=$ $Sc(G;P)$. By Lemma \ref{regtrivialsource} we have
$$\reg(H^*(G;k); \Ext^*_{kG}(k; V)) \leq 0.$$
Since $V$ has a nonzero trivial submodule, we have $H_0(G;V)\neq \{0\}, $
and hence the other inequality follows from Lemma \ref{regnonzerohomology}.
\end{proofscottreg}

Theorem \ref{HHregzero} will be a consequence of Theorem \ref{Scottregzero} 
and the following well-known observation (for which we include a proof for the
convenience of the reader; the block theoretic background material
can be found in \cite{Thev}).

\begin{lemma} \label{BScott}
Let $G$ be a finite group, $B$ a block algebra of $kG$ and
$P$ a defect group of $B$. As a module over $kG$ with respect
to the conjugation action of $G$ on $B$, the $kG$-module $B$ has
an indecomposable direct summand isomorphic to the Scott module
$Sc(G;P)$.
\end{lemma}

\begin{proof}
Since the conjugation action of $G$ on $B$ induces the trivial action
on $Z(B)$ and since $Z(B)\neq$ $\{0\}$, it follows that the $kG$-module 
$B$ has a nonzero trivial submodule. Moreover, $B$ is a direct summand 
of $kG$, hence $B$ is a $p$-permutation $kG$-module, and the vertices 
of the indecomposable direct summands of $B$ are conjugate to subgroups 
of $P$. Thus $B$ has a Scott module with a vertex contained in 
$P$ as a direct summand. Since $Z(B)$ is not contained in the kernel of 
the Brauer homomorphism $\Br_P$, it follows that $B$ has a direct 
summand isomorphic to the Scott module $Sc(G;P)$. 
\end{proof}

\begin{proofhhreg}
By \cite[Proposition 5]{KeLiHHgen} we have $\reg(HH^*(B))\leq$ $0$. Recall that   $HH^*(kG) $  is   an  $H^*(G;k)$-module via the 
diagonal induction map, and we   have a canonical 
graded  isomorphism $HH^*(B)\cong$ $H^*(G;B)$  as $H^*(G;B)$-modules  where $G$ 
acts on $B$ by conjugation; see e. g. \cite[(3.2)]{SiWi}.
It follows from \cite[Lemma 4]{KeLiHHgen} that
$$\reg(HH^*(B)) = \reg(H^*(G;k); H^*(G; B))\ .$$
By Lemma \ref{BScott}, the $kG$-module $B$ has a direct summand isomorphic
to $V=$ $Sc(G;P)$, where $P$ is a defect group of $B$.
Thus as an $H^*(G;k)$-module, $H^*(G;B)$ has a direct summand isomorphic to
$H^*(G;V)$. It follows that 
$$\reg(HH^*(B))\geq \reg(H^*(G;k); H^*(G;V))=0\ ,$$ 
where the last equality is from Theorem \ref{Scottregzero}.  This completes
the proof of Theorem \ref{HHregzero}.
\end{proofhhreg}

\begin{re} 
The above proof can be adapted to show that the regularity of the 
stable quotient $\overline{HH^*}(B)$ of $HH^*(B)$  also equals zero. 
Recall that $\overline{HH^*}(B)$ is the quotient of $HH^*(B)$ by the 
ideal $ Z^{\pr}(B) = $ $\Tr_{1}^{G}(B)$ of $Z(B) \cong$ $HH^0(B)$.   
Note that $Z^{\pr}(B)$ is concentrated in degree $0$. Alternatively,    
$\overline{HH^*}(B)$ may be defined as the non-negative part of the 
Tate Hochschild cohomology of  $B$. Our  interest in  
$\overline{HH^*}(B)$ comes from the fact that  Tate Hochschild 
cohomology of symmetric algebras is an invariant of stable equivalence 
of Morita type.  We briefly indicate how the regularity of 
$\overline{HH^*}(B) $  may be calculated.  
Let $B =\oplus_i M_i $ be a decomposition of $B$ into a direct sum of
indecomposable $kG$-modules $M_i$, where  $G$ acts by conjugation on 
$B$.  The  canonical graded $H^*(G; k)$-module isomorphism 
$HH^*(B)\cong$ $H^*(G;B)$ induces an isomorphism  
$$HH^0(B)\cong H^0(G;B) = \oplus_i\  H^0(G;  M_i)$$
in degree zero. Composing this with the the canonical isomorphisms 
$Z(B) \cong$ $HH^0(B)$ and $H^0(G;  M_i)  \cong$ $M_i ^G $, it is easy 
to check that the image of $Z^{\pr}(B)$ in $\oplus_i  M_i^G$ is 
$\oplus_i \Tr_{1}^{G} (M_i) $. Since  $B$ is a $p$-permutation 
$kG$-module, $\Tr_{1}^{G}(M_i)$ is non-zero precisely if $M_i$ is    
isomorphic to the Scott module $Sc(G;1)$ (which is a projective cover 
of the trivial $kG$-module). Let $M'$ denote the sum of all $M_i $'s 
in the above decomposition which are isomorphic to $Sc(G,1)$ and let $M''$ 
be the complement of $M'$ in $B$ with respect to the above 
decomposition. Since $Z^{pr} (B)$ is concentrated in degree zero,  
we have a direct sum decomposition  $HH^*(B)\cong$ 
$\oplus H^*(G; M'') \oplus Z^{\pr}(B)$ as $H^*(G; k)$-modules. 
In particular,
$$ \reg(H^*(G;k);  HH^*(B)) =  
\max\{\reg(H^*(G;k); H^*(G; M'') ),   \reg (H^*(G;k); Z^{pr}(B)) \}. $$    
We  may  assume that    a defect group $P$   of $B$ is non-trivial.  By Lemma 
\ref{BScott},  $M'' $ contains a direct 
summand isomorphic to $Sc(G;P)$.  Hence  by  Theorem \ref{Scottregzero}     
$\reg(H^*(G;k); H^*(G; M''))  \geq 0 $.  It follows from Theorem 
\ref{HHregzero}  and the above displayed equation  that $\overline{HH^*}(B)\cong$ $H^*(G; M'')$  has 
regularity zero.
\end{re}

\begin{prooffreg}
By \cite[Proposition 6.1]{Symreg} we have $\reg(H^*(P;k)^\CF)\leq$ $0$.
For the other inequality we follow the arguments in 
\cite[\S 3, \S 4]{BenInjCohom}, applied to transfer maps using
fusion stable bisets. 
For $Q$ a subgroup of $P$ and $\varphi : Q\to$ $P$ an injective group
homomorphism, we denote by  $P\times_{(Q,\varphi)} P$ 
the $P$-$P$-biset of equivalence classes in  $P\times P$ with respect
to the relation $(uw,v) \sim $ $(u,\varphi(w)v)$, where $u, v\in$ $P$, 
and $w\in$ $Q$. The $kP$-$kP$-bimodule having  $P\times_{(Q,\varphi)} P$ 
as a $k$-basis is canonically isomorphic to $kP\tenkQ (_\varphi{kP})$.
This biset gives rise to a transfer map 
$\tr_{P\times_{(Q,\varphi)} P}$ 
on $H^*(P;k)$ obtained by composing the restriction map
$\res^P_{\varphi(Q)} : H^*(P;k)\to$ $H^*(\varphi(Q);k)$, the
isomorphism $H^*(\varphi(Q);k)\cong$ $H^*(Q;k)$ induced by $\varphi$, 
and the transfer map $\tr^P_Q : H^*(Q;k)\to$ $H^*(P;k)$.
Let $X$ be an $\CF$-stable $P$-$P$-biset
satisfying the conclusions of \cite[Proposition 5.5]{BLO}. That is, 
every transitive subbiset of $X$ is isomorphic to
$P \times_{(Q,\varphi)} P$ for some subgroup $Q$
of $P$ and some group homomorphism $\varphi : Q \to P$
belonging to $\CF$, the integer $\vert X\vert / \vert P\vert$ is prime 
to $p$, and for any subgroup $Q$ of $P$ and any group homomorphism
$\varphi : Q \to P$ in $\CF$, the $Q$-$P$-bisets 
${}_\varphi X$ and ${}_QX$ (resp. the $P$-$Q$-bisets $X_Q$ and $X_\varphi$)
are isomorphic. By taking the sum, over the 
transitive subbisets $P\times_{(Q,\varphi)} P$, of the transfer maps 
$\tr_{P\times_{(Q,\varphi)} P}$, we obtain a transfer map $\tr_X$ on 
$H^*(P;k)$. Following \cite[Proposition 3.2]{LinSchur},
the map $\tr_X$ acts as multiplication by $\frac{|X|}{|P|}$ on
$H^*(P;k)^\CF$, hence $\Im(\tr_X)=$ $H^*(P;k)^\CF$, and we have 
a direct sum decomposition
$$H^*(P;k) = H^*(P;k)^\CF \oplus \ker(\tr_X)$$
as $H^*(P;k)^\CF$-modules. A similar decomposition holds for Tate
cohomology, and for homology (using either the canonical duality
$H_n(P;k)\cong$ $H^n(P;k)^\vee$ or the isomorphism $H_n(P;k)\cong$ 
$\hat H^{-n-1}(P;k)$ obtained from composing the previous duality
with Tate duality). By \cite[Equation (4.1)]{BenInjCohom},
the transfer map $\tr^P_Q$ induces a homomorphism of Greenlees'
local cohomology spectral sequences
$$\xymatrix{ H_m^{i,j} H^*(Q,k) \ar[d]_{(\tr^P_Q)_*} \ar@{=>}[r] 
&  H_{-i-j}(Q;k) \ar[d]^{(\res^P_Q)_*} \\
H^{i,j}_m H^*(P;k) \ar@{=>}[r]   & H_{-i-j}(P;k) }$$
where $(\tr^P_Q)_*$ and $(\res^P_Q)_*$ are the maps induced
by $\tr^P_Q$ and the inclusion $Q\to$ $P$, respectively.
The isomorphism $\varphi : Q\to$ $\varphi(Q)$ induces an
obvious isomorphism of spectral sequences
$$\xymatrix{ H_m^{i,j} H^*(\varphi(Q),k) \ar[d]_{\cong} \ar@{=>}[r] 
&  H_{-i-j}(\varphi(Q);k) \ar[d]^{\cong} \\
H^{i,j}_m H^*(Q;k) \ar@{=>}[r]  & H_{-i-j}(Q;k) }$$
Restriction and transfer on Tate cohomology are dual to each other
under Tate duality, and hence the dual version of 
\cite[Equation (4.1)]{BenInjCohom} implies that the restriction 
$\res^P_{\varphi(Q)}$ induces a homomorphism of spectral sequences
$$\xymatrix{ H_m^{i,j} H^*(P,k) \ar[d]_{(\res^P_{\varphi(Q)})_*} \ar@{=>}[r] 
 & H_{-i-j}(P;k) \ar[d]^{(\tr^P_{\varphi(Q)})_*} \\
H^{i,j}_m H^*(\varphi(Q);k) \ar@{=>}[r]  & H_{-i-j}(\varphi(Q);k) }$$
Composing the three diagrams above yields a homomorphism
induced by $\tr_{P\times_{(Q,\varphi)} P}$ on the spectral
sequence for $P$, and taking the sum over all transitive subbisets
of $X$ yields a homomorphism of spectral sequences
$$\xymatrix{ H_m^{i,j} H^*(P,k) \ar[d]_{(\tr_X)_*} \ar@{=>}[r] 
 & H_{-i-j}(P;k) \ar[d]^{(\tr_{X^\vee})_*} \\
H^{i,j}_m H^*(P;k) \ar@{=>}[r] &  H_{-i-j}(P;k) }$$
where $X^\vee$ is the $P$-$P$-biset $X$ with the opposite
action $u\cdot x\cdot v=$ $v^{-1}xu^{-1}$ for all $u,v\in$ $P$ and $x\in$ $X$.
One easily checks that $X^\vee$ is isomorphic to a dual basis of $X$
in the dual bimodule $\Hom_k(kX,k)$.
By \cite[Proposition 5.2]{BLO}, $H^*(P;k)$ is finitely generated as a 
module over $H^*(P;k)^\CF$. Thus the local cohomology spaces
$H^{i,j}_m H^*(P;k)$ can be calculated using for $m$ the maximal
ideal of positive degree elements in $H^*(P;k)^\CF$ instead of $H^*(P;k)$.
It follows that $\tr_X$ induces a homomorphism of spectral sequences
$$\xymatrix{ H_m^{i,j} H^*(P,k) \ar[d]_{(\tr_X)_*} \ar@{=>}[r] 
 & H_{-i-j}(P;k) \ar[d]^{(\tr_{X^\vee})_*} \\
H^{i,j}_m H^*(P;k)^\CF \ar@{=>}[r] &  H_{-i-j}(P;k)^\CF }$$
For $i= - j = r$, where $r$ is the rank of $P$, the edge homomorphism
yields a commutative diagram of the form
$$\xymatrix{ H^{r,-r}_m H^*(P;k) \ar[rr]^{\gamma_P} \ar[d]_{(\tr_X)_*}
& & H_0(P;k) \ar[rr]^{\cong} \ar[d]_{(\tr_{X^\vee})_*} 
& & k \ar[d]^{\cdot \frac{|X|}{|P|}} \\
H^{r,-r}_m H^*(P;k)^\CF \ar[rr]_{\delta_\CF} & & H_0(P;k)^\CF
\ar[rr]_{\cong} & & k }$$
where the right vertical map is multiplication on $k$ by 
$\frac{|X|}{|P|}$. By \cite[Theorem 4.1]{BenInjCohom}, the map
$\gamma_P$ is surjective, and hence so is the map $\delta_\CF$.
In particular, $H^{r,-r}_m H^*(P;k)^\CF\neq$ $\{0\}$, whence the
result.
\end{prooffreg}

\begin{re} The fact that transfer and restriction on
Tate cohomology are dual to each other under Tate duality
can be deduced from a more general duality for transfer 
maps on Tate-Hochschild cohomology of symmetric algebras 
induced by bimodules which are finitely generated projective 
as left and right modules (cf. \cite{LinTate}).  
\end{re}

\end{document}